\documentclass[a4paper]{amsart}
\usepackage[utf8]{inputenc}
\usepackage[british]{babel}
\usepackage{amsmath,amssymb}
\usepackage{bm}
\usepackage{mathrsfs}
\usepackage{verbatim}
\usepackage{tikz}
\usepackage{esint}
\usepackage{siunitx}
\usepackage{booktabs}

\newtheorem{myDef}{Definition}
\newtheorem{myTheo}[myDef]{Theorem}
\newtheorem{mylem}[myDef]{Lemma}
\newtheorem{corollary}[myDef]{Corollary}
\newtheorem{remark}[myDef]{Remark}

\numberwithin{equation}{section}
\numberwithin{myDef}{section}

\begin{document}

\title[Continuous elements satisfying MT]
     {Continuous finite elements satisfying 
     a strong discrete Miranda--Talenti identity}

\author [D. Gallistl \and S. Tian]{Dietmar Gallistl \and Shudan Tian}
\address[D. Gallistl]{Friedrich-Schiller-Universit\"at Jena, Institut f\"ur Mathematik,
         Ernst-Abbe-Platz 2, 07743 Jena, Germany}
\email{dietmar.gallistl [at] uni-jena.de}
\address[S. Tian]{School of Mathematics and Computational Science,
         Xiangtan University, Xiangtan 411105, China}
\email{shudan.tian [at] xtu.edu.cn}

\thanks{The first author is supported by the 
European Research Council
(ERC Starting Grant \emph{DAFNE}, agreement ID 891734).
The second author acknowledges support by 
Sino-German (CSC-DAAD) Postdoc Scholarship Program, 2021 ID 57575640%
}



\begin{abstract}
This article introduces continuous $H^2$-nonconforming finite
elements in two and three space dimensions
which satisfy a strong discrete Miranda--Talenti
inequality in the sense that the global $L^2$ norm of the 
piecewise Hessian is bounded by the $L^2$ norm of the piecewise
Laplacian.
The construction is based on globally continuous finite element
functions with $C^1$ continuity on the vertices (2D)
or edges (3D).
As an application, these finite elements are used to approximate
uniformly elliptic equations in non-divergence form under
the Cordes condition without additional
stabilization terms.
For the biharmonic equation in three dimensions,
the proposed methods has less degrees of freedom
than existing nonconforming schemes of the same order.
Numerical results in two and three dimensions confirm
the practical feasibility of the proposed schemes.
\end{abstract}

\maketitle

\section{Introduction}

It is a classical result that,
for any bounded convex domain $\Omega\subseteq\mathbb R^d$, 
the Hessian of any function $v\in H^1_0(\Omega)\cap H^2(\Omega)$
can be bounded, in the $L^2$ norm, by the Laplacian:
\begin{equation}
 \|D^2 v\|_0\leq \|\triangle v\|_0
 \quad\text{for all } v\in H^1_0(\Omega)\cap H^2(\Omega).
 \label{MT}
\end{equation}
Here and throughout this article, the spaces $H^k(\Omega)$
with $k\geq 0$ denote the usual $L^2$-based Sobolev spaces,
(the subindex $0$ refers to a homogeneous Dirichlet boundary
condition in the sense of traces)
and $\|\cdot\|_0$ is the $L^2$ norm over $\Omega$.
Estimate \eqref{MT} is known as the
\emph{Miranda--Talenti inequality},
named after the authors of \cite{Miranda} and \cite{Talenti},
and proofs thereof can be found in \cite{bihar-regular,Suli2013}.
If $\Omega$ is polytopal (not necessarily convex),
estimate \eqref{MT} becomes even an identity.
The Miranda--Talenti estimate is critical for the well-posedness
of second-order partial differential equations (PDEs) in
nondivergence form
\begin{equation}
A:D^2u \equiv \sum_{i,j=1}^d A_{ij}\partial^2_{ij} u = f
\quad \text{in } \Omega,\quad u=0 \text{ on }\partial\Omega
\label{nondiv}
\end{equation}
where $f\in L^2(\Omega)$ is a given right-hand side and the 
coefficient $A:\Omega\to\mathbb R^{d\times d}$ is essentially
bounded, uniformly elliptic, and satisfies the Cordes condition
(details follow in Section~\ref{s:nondiv}),
see \cite{MT,Suli2013}.
In this setting,
Galerkin methods in subspaces of $H^1_0(\Omega)\cap H^2(\Omega)$
are automatically well-posed because the discrete functions involved
enjoy \eqref{MT}. On the other hand, such subspaces are generally
cumbersome already in the planar case $d=2$ while they are considered
completely impractical for $d=3$.
For example, polynomial $H^2$-conforming elements in 3D need at least 
220 local degrees of freedom.
A classical alternative are nonconforming finite element spaces,
which are not subspaces of $H^2(\Omega)$ but rather involve weaker
continuity constraints on the discrete functions.
If, for example, the discrete functions are chosen continuous but
only with certain weak continuity conditions for the second-order
derivative, much simpler local constructions are possible.
In this paper, we are interested in nonconforming methods
satisfying the following discrete version of \eqref{MT}.

\begin{myDef}[strong discrete Miranda--Talenti property]\label{def:dmt}
Let $\mathcal T$ be a simplicial triangulation of $\Omega$
and let $V_{\mathcal T}$ be a space of piecewise polynomial functions
with respect to $\mathcal T$.
We say that $V_{\mathcal T}$ satisfies the 
strong discrete Miranda--Talenti property, if 
	\begin{equation}
	 \|D^2_{\mathcal T}v_h\|_0\leq \|\triangle_{\mathcal T} v_h\|_0
            \quad \text{for all }v_h\in V_{\mathcal T}.
		\label{dMT}
	\end{equation}
Here $D^2_{\mathcal T}v_h$ and $\triangle_{\mathcal T}v_h$ denote the
piecewise (with respect to $\mathcal T$) evaluation of
the Hessian and the Laplacian of the piecewise smooth function
$v_h$, respectively.
\end{myDef}

The reasoning in nonconforming finite elements (as opposed to
discontinuous Galerkin methods) is to impose certain continuity constraints
that make the use of the piecewise regular part of the differential operator
sufficient for a method to work without additional stabilization or
penalization. The validity of \eqref{dMT} additionally requires certain
compatibility conditions that allow for an integration by parts
(see Lemma~\ref{DMT1} below).
For equations of the biharmonic type, such as Kirchhoff plates,
nonconforming methods are a well established tool \cite{Ciarlet2002},
but they are usually unstable when applied to problems in
nondivergence form; the reason being that their
derivatives (piecewise with respect to some simplicial
triangulation of the domain)
do not satisfy \eqref{dMT}.
A simple example for this phenomenon is the Morley finite element
\cite{Ciarlet2002},
some basis functions of which are nonzero,
piecewise quadratic, and have normal derivatives with vanishing 
jump integrals across all $(d-1)$-dimensional hyperfaces.
The divergence theorem therefore implies that the
piecewise Laplacian (which is piecewise constant)
of such a function must vanish, making a discrete version of
\eqref{MT} impossible.

There exist several discretizations of \eqref{nondiv} that do
not require \eqref{dMT} as a discrete strong version
of \eqref{MT} and instead
involve stabilization terms:
the discontinuous Galerkin finite element method
of \cite{Suli2013},
mixed methods \cite{Mixednondiv},
methods based on
the discrete Miranda--Talenti inequality from \cite{MTlag},
and the stabilized Hermite element 
of \cite{Wuhermite}.
In this paper we are concerned with constructing
nonconforming methods that satisfy \eqref{dMT}
and are thus suited for approximating \eqref{nondiv}
without any stabilization term in the method.
We note that, in two dimensions, there exists a nonconforming
scheme satisfying a strong discrete Miranda--Talenti inequality
\cite{MTzhang}, which ---similar to the Fortin--Soulie element
\cite{FortinSoulie}--- is not a finite element in the sense of
local degrees of freedom.
In three space dimensions, we are not aware of any nonconforming
method satisfying \eqref{dMT}.
This article focus on constructing nonconforming elements satisfying 
discrete Miranda--Talenti inequality.
We contribute results for two and three space dimensions.
The construction is motivated by the $H^2$-nonconforming elements
which are $C^0$-elements with $C^1$ continuity on $(d-2)$-dimensional
hyperfaces (vertices if $d=2$; edges if $d=3$).
The latter turns out to be an important ingredient in the
construction but it is not sufficient,
as can be seen, e.g., from
the Adini element \cite{finiteBrenner,Ciarlet2002,finitewang} 
and the first-order Specht element \cite{Specht1ord}.
In two dimensions, we first prove that the 
second-order Specht element \cite{specht2,specht3}
satisfies \eqref{dMT} (see Lemma~\ref{Specht}).
This element belongs is $C^0$ continuous
and has $C^1$ continuity in the element vertices.
We then propose a high order family of 
$H^2$-nonconforming finite elements for $d=2$,
which is defined for any odd approximation order.
In the case $d=3$ we propose a family of nonconforming
finite elements satisfying \eqref{dMT}.
It is worth mentioning that, in the three-dimensional case,
our new element seems to be the first such element
in the literature.
The approximation of the lowest-order version
in the $H^2$ norm is of fourth order.
Since in three dimensions the $C^1$-continuity across edges 
requires polynomial shape functions to be of at least fifth order,
our element appears to be in some sense minimal,
whence we call it a low-order method.
It also has less degrees of freedom compared to other
known $H^2$-nonconforming elements of similar order
\cite{Guzm2012,HZT}.

The organisation of this paper as follows. 
In Section~\ref{s:prelim}, we introduce some notation
 and the strong discrete Miranda--Talenti inequality. 
The proof of Specht element satisfying discrete Miranda--Talenti inequality and 
designing 2D odd order elements are presented in 
Section~\ref{s:2d}.
In Section~\ref{s:3d}, we construct the 3D satisfying 
discrete Miranda--Talenti property $H^2$-nonconforming elements. 
Section~\ref{s:bih} and Section~\ref{s:nondiv} present
the application the the biharmonic equation and 
elliptic problems in nondivergence form.
In the Section~\ref{s:num} we present numerical experiments,
followed by the concluding Section~\ref{s:conclusion}.

\section{Preliminaries}\label{s:prelim}

\subsection{General notation}
Throughout this paper, $\Omega\subset \mathbb{R}^d$, $d=2,3$ is an open and bounded
domain with polytopal Lipschitz boundary $\partial\Omega$.
The standard Sobolev spaces are denoted by $H^m(\Omega)$, $m\geq 0$, with
$H^0(\Omega)=L^2(\Omega)$.
The space of tensor-valued functions with all components in $H^m(\Omega)$
is denoted by $H^m(\Omega;\mathbb{R}^{d\times d})$.
The space $H_0^1(\Omega)$ is the subspace of $H^1(\Omega)$ of functions
with vanishing trace on $\partial\Omega$.
For any measurable subdomain $G\subseteq \Omega$,
the $L^2$ inner product is denoted by
$(\cdot,\cdot)_G$.
The symbols $\nabla$ and $D^2$ denote the gradient and
the Hessian, respectively.
The space of polynomial functions over $G\subseteq \Omega$ of degree not greater than $\ell$
is denoted by $P_{\ell}(G)$.
Given a unit vector $\xi\in\mathbb R^d$, the directional derivative
with respect to $\xi$
is denoted by $\partial_\xi u= \nabla u \cdot \xi$.

\subsection{Meshes and discrete functions}\label{ss:meshes}
Let $\mathcal{T}$ be a shape-regular simplicial triangulation of 
$\Omega$, and $h_T$ be the diameter of the element $T\in\mathcal T$.
The global mesh size reads $h=\max_{T\in\mathcal{T}}\{h_T\}$.
Let $H^m(\mathcal{T})$ be the spaces of functions that belong to
$H^m(T)$ when restricted to the interior 
of any simplex $T\in\mathcal T$
(we follow the convention to not explicitly indicate the interior
of the closed domain $T$ in the notation of Sobolev spaces).
Let $\mathcal F$ denote the set of all $(d-1)$-dimensional hyperfaces
of the triangulation $\mathcal T$.
For brevity, we will sometimes write ``face'' instead of 
``$(d-1)$-dimensional hyperface''; in particular ``face'' and ``edge''
are synonymous in the planar case $d=2$.
To each face $F\in\mathcal F$ we assign a fixed unit normal vector
$n_F$ with the convention that $n_F$ coincides with the outward-pointing
unit vector to $\Omega$ if $F\subset\partial\Omega$ is a boundary face.
Given any interior face $F\in\mathcal F$ shared by two elements
$T_+,T_-$, the normal derivative jump of a function $v$ across $F$ is defined as
$$
[\partial_n v]|_F:=\partial_{n_F}v|_{F\cap T_+} - \partial_{n_F}v|_{F\cap T_-}
$$
where the vector $n_F$ is the normal pointing from $T_+$
towards $T_-$; it coincides with the outer unit normal to
$\partial_{T_+}$ on $F$ and to the inner unit normal to $\partial_{T_-}$
on $F$.
If $F\subset \partial \Omega$ is a boundary face,
$[\partial_nu]|_F$ 
denotes the normal derivative restricted to $F$.
The definition of the jump does not depend on the choice of sign for 
$n_F$.
We will denote 
the integration of the normal derivative of $v$ on 
$(d-1)$-dimensional sides by $\int_F\partial_nv\mathrm{d}S$.
For a piecewise polynomial function $v$, we say that the normal
derivative enjoys $k$th order moment continuity across the face $F$
if
$\int_F [\partial_n v] p_k\,dS = 0$ for all $p_k\in P_k(F)$.
The piecewise versions of the Laplacian and the Hessian
are defined as
$\triangle_{\mathcal T} v :=\sum_{T \in \mathcal{T}} \triangle v|_T$,
$D^2_{\mathcal T} v :=\sum_{T \in \mathcal{T}} D^2 v|_T$
where the convention is used that $\triangle v|_T$
and $D^2 v|_T$ are extended to the whole domain $\Omega$
by zero.
These piecewise differential operators do not consider any 
inter-element jumps (as their distributional interpretation would to).

\subsection{An identity from integration by parts}

Our construction of nonconforming elements which satisfy the strong discrete 
Miranda-Talenti property \eqref{dMT}
will be based on $C^0$-continuous elements with $C^1$ continuity across 
$(d-2)$-dimensional hyperfaces. A proof of 
the following identity based on integration
by parts can be found in \cite{Wuhermite}.
\begin{mylem}
Let $\mathcal{T}$ be a regular triangulation of 
$\Omega\subset \mathbb{R}^d,d=2,3$ and $V_{\mathcal T}\subseteq H^2(\mathcal T)$
be a $C^0$ finite element space that has $C^1$ continuity on $(d-2)$-dimensional 
hyperfaces.
Then any $v_h\in V_{\mathcal T}$ satisfies
\begin{equation*}
\|\triangle_{\mathcal T}v_h\|_0^2 
 = \|D^2_{\mathcal T}v_h\|_0^2+2\sum_{i=1}^{d-1}\sum_{F\in \mathcal{F}}([\partial_n v_h],\partial^2_{t_{F_i}}v_h)_{F}.
\end{equation*}
Here $(t_{F_i})_{i=1}^{d-1}$
is any orthonormal set of tangential vectors to 
$F\in\mathcal F$.
\label{DMT1}
\end{mylem}

\section{Design of the finite element in two dimensions}\label{s:2d}
 
We will work with shape function spaces $\mathcal{P}_{T}$ 
with respect to an element $T\in\mathcal T$ that
consist of the full polynomial space of degree $\ell$
plus a space $B_{T,\ell}$ of some particular bubble functions, i.e.,
\begin{equation} 
\mathcal{P}_{T} = P_{\ell}(T)+B_{T,\ell}. 
\label{SFspace}
\end{equation}
The careful choice of the space 
$B_{T,\ell}\subseteq H^1_0(T)$
turns out crucial for the desired property \eqref{dMT}.

\subsection{Notation}
We use the following convention on enumeration of vertices and 
edges of a triangle.
The three vertices of a triangle are denoted by
$p_i,i=1,2,3$ 
and $F_i,i=1,2,3$ are the three edges.
We assume a counter-clockwise enumeration as in Figure~\ref{f:triangle}.
In particular, $F_i$ is the convex combination of 
$p_{i-1}$ and $p_{i+1}$.
Here and throughout this section,
we always omit $\mathrm{mod}3$ in the subscript of $p_{i-1}$ and $p_{i+1}$. 
The functions
$\lambda_i,i=1,2,3$ are the three barycentric coordinates on the element, i.e., $\lambda_i$ is an
affine function on $T$ satisfying $\lambda_i(a_j)=\delta_{i,j},~ 1\leq i,j\leq 3$.
We define the non-negative cubic volume bubble function as
$b_T=\lambda_1\lambda_2\lambda_3$ 
and the quadratic edge bubble functions by $b_{F_i}=\lambda_{i-1}\lambda_{i+1}$.

	\begin{figure} \setlength\unitlength{1pt}
	\begin{center} 
	\begin{tikzpicture}
\draw (0,0) -- (2,0);	
\draw (0,0) -- (1,1.732);
\draw (2,0) -- (1,1.732);
\node [below] at (1,0) {$F_1$};
\node [above] at (1,1.732) {$p_1$};
\node [right] at (2,0) {$p_3$};
\node [left] at (0,0) {$p_2$};	
\node [left] at (0.5,1.732/2) {$F_3$};
\node [right] at (1.5,1.732/2) {$F_2$};
\draw[fill] (1,1.732) circle [radius=0.05];
\draw[fill] (2,0) circle [radius=0.05];
\draw[fill] (0,0) circle [radius=0.05];
	\end{tikzpicture}
\end{center}
\caption{Enumeration of vertices and edges in a triangle.
         \label{f:triangle}}
\end{figure}

\subsection{Review of the second-order Specht element}
References
\cite{specht,specht2,specht3} introduce the second-order Specht element, 
which has the convergence order $2$ under optimal regularity
assumptions.
The local shape functions on an element $T\in\mathcal T$ are as follows
$$
 \mathcal{P}_T=P_3(T)+\mathrm{span}\{\phi_i:i=1,2,3\},
$$
where the three functions $\phi_i$ are defined as
$$
\phi_i = 2b_T\big(5(\lambda_i-\lambda_i^2-2\lambda_{i-1}\lambda_{i+1})-1\big),
~1\leq i\leq 3
.
$$
The space of functions over $\Omega$ that belong to $\mathcal P_T$
when restricted to any $T\in\mathcal T$ is denoted by $\mathcal P_{\mathcal T}$.
The global finite element space $V_{\mathcal T}$ reads
$$
\begin{aligned}
V_{\mathcal T} := 
H^1_0(\Omega)\cap 
\left\{
   v \in \mathcal P_{\mathcal T } :
   \begin{aligned}
   &v\text{ and }\nabla v \text{ are continuous at all vertices}
   \\
   &\text{and }
         \textstyle\int_F [\partial_n v]\,ds = 0 \text{ for any interior edge } F
\end{aligned}
\right\}
.
\end{aligned}
$$
It is direct to verify that the functions of $V_{\mathcal T}$ are continuous
with $C^1$ continuity at the vertices.
The next result establishes the strong discrete Miranda--Talenti
property \eqref{dMT} for
the second-order Specht element.

\begin{mylem}
Any $v_h\in V_{\mathcal T}$ satisfies
$$
\|\triangle_{\mathcal T}v_h\|_0^2 = \|D^2_{\mathcal T}v_h\|_0^2.
$$ 
\label{Specht}
\end{mylem}

\begin{proof}
In the present two-dimensional case, we assign to any edge $F$ the unit tangent
$(-n_{F,2},n_{F,1})$ and abbreviate the tangential derivative
by $\partial_t$.
	According to Lemma \ref{DMT1}, the continuity of $v_h$ over interior edges
	and the continuity of $\nabla v_h$ at the vertices imply
	 \begin{align*}
\|\triangle_{\mathcal T}v_h\|_0^2 = \|D^2_{\mathcal T}v_h\|_0^2+2\sum_{F\in \mathcal{F}}([\partial_n v_h],\partial^2_{t}v_h)_{F}.
	\end{align*}
 Let $F_i\in\mathcal F$ be an interior edge
  with endpoints
 $p_{i-1}$ and $p_{i+1}$ and normal vector $n=n_{F_i}$.
 By the defining degrees of freedom,
 the integral over $F_i$ of the normal derivative jump
 of $v_h$ against constants vanishes.
 Therefore the second-order tangential derivative
 and its integral mean $m$ over $F_i$ satisfy
$$
 \int_{F_i}\partial^2_{t}v_h[\partial_nv_h]\mathrm{d}s =
  \int_{F_i} (\partial^2_{t}v_h - m ) [\partial_nv_h] \mathrm{d}s
  .
$$ 
 We note that
 $\partial^2_{t}v_h|_{F}$ is an affine function along $F_i$,
 whence $\partial^2_{t}v_h - m$ vanishes
 in the midpoint of $F_i$ and thus is a multiple of $\lambda_{i+1}-1/2$.
 Therefore there exists a constant $c$ such that
\begin{equation}
 \int_{F_i}\partial^2_{t}v_h[\partial_nv_h]\mathrm{d}s =
  c\int_{F_i} (\lambda_{i+1}-1/2) [\partial_nv_h] \mathrm{d}s
  .
  \label{e:simpson_prep}
\end{equation}
We next claim the identity
\begin{equation}
\fint_{F_i}\lambda_{i+1} [\partial_{n}v_h]\mathrm{d}s=
\frac{1}{12}([\partial_{n}v_h](p_{i+1})-[\partial_{n}v_h](p_{i-1}))
   +\frac{1}{2}\fint_{F_i}[\partial_{n}v_h]\mathrm{d}s
\label{simpson}
\end{equation}
where $[\cdot]$ always denotes the jump across $F_i$.
The proof of the identity utilizes the fact that $\lambda_{i+1}-1/2$
vanishes in the midpoint of $F_i$ and takes the value $\pm1/2$ at the 
vertices so that the result for the piecewise cubic part of $v_h$
follows from an application of the Simpson
quadrature rule, which is exact for polynomials of degree $3$.
Since the normal derivatives of the bubble functions $\phi_i$
have vanishing integrals against
multiples of $\lambda_i -1/2$ over the edges 
(which can be verified by direct computations)
and vanish in the vertices, the result holds for any $v_h\in V_{\mathcal T}$.
The combination of \eqref{e:simpson_prep} and \eqref{simpson}
shows
$$
 \int_{F_i}\partial^2_{t}v_h[\partial_nv_h]\mathrm{d}s = 0
$$
because the normal derivative jumps vanish in the vertices
due to the defining constraints of $V_{\mathcal T}$.
\end{proof}
\subsection{Extension of the second-order Specht element to any odd order}\label{section2D}
In this section, we will extend the Specht element to any odd order in 2 dimensions.
For any triangle $T\in\mathcal T$ and $i=1,2,3$
we define degrees of freedom (linear functionals) 
mapping any $v\in C^\infty(\mathbb R^2)$ to
\begin{subequations}
\begin{align}
    \label{e:dof2d_nodal_edge}
	v(a_i) \text{ and }
	\nabla v(a_i)
    \text{ and }
	\fint_{F_i}v q_{\ell-4}\mathrm{d}s \text{ for any }q_{\ell-4}\in P_{\ell-4}(F_i),\\
		  	\label{e:dof2d_normalA}
	\fint_{F_i}\partial_nv q_{\ell-4}\mathrm{d}s 
	\text{ for any }q_{\ell-4}\in P_{\ell-4}(F_i),\\
	\label{e:dof2d_normalB}
	\fint_{F_i}\partial_nv \,\lambda_{i+1}^{\ell-3}\mathrm{d}s,
	 \\
    \label{e:dof2d_volume}
	\fint_T v q_{\ell-6} \mathrm{d}x , 
	\text{ for any }q_{\ell-6} \in P_{\ell-6}(T).
\end{align}
	 \end{subequations}
If $\ell<6$, we interpret \eqref{e:dof2d_volume} as being void.
The dimension of the space
$\Sigma_{T,\ell}$ spanned by the linear functionals 
\eqref{e:dof2d_nodal_edge}--\eqref{e:dof2d_volume} is
at most $\mathrm{dim}P_{\ell}(T)+3$. 

In what follows we will say that a function defined over an edge $F$
is odd (resp.\ even), if, after an affine change of coordinates where $F$ is
mapped to the real interval $[-1,1]$, the transformed function is
odd (resp.\ even).

\begin{mylem}
Let $\ell\geq 4$ be even.
Then,
the degrees of freedom \eqref{e:dof2d_nodal_edge}, 
\eqref{e:dof2d_normalA}, \eqref{e:dof2d_volume}
are linear independent as functionals over
$P_{\ell}(T)$.
\label{uni2D}
\end{mylem}

\begin{proof}
	Let $v_h\in P_{\ell}(T)$ such that all degrees of freedom
	from \eqref{e:dof2d_nodal_edge}, \eqref{e:dof2d_normalA}, \eqref{e:dof2d_volume}
	vanish on $v_h$.
	We will prove $v_h\equiv 0$.
	Since the functionals \eqref{e:dof2d_nodal_edge} vanish on $v_h$,
	the values of $v_h$ and $\nabla v_h$ vanish at all vertices
	$a_i$ and all edge moments $\int_{F_i}v q_{\ell-4}\mathrm{d}s$,
	against polynomials $q_{\ell-4}$ of degree at most $\ell-4$
	are zero.
	Thus $v_h$ vanishes on the boundary of $T$ and there exists some 
	$g_{\ell-3}\in P_{\ell-3}(T)$ such that
	$$
	v_h = b_T g_{\ell-3}.
	$$
	Since the functionals
	\eqref{e:dof2d_normalA} vanish,
	we have for all $q_{\ell-4}\in P_{\ell-4}(F_i)$ that
	$$
	0=\int_{F_i}\partial_nv_hq_{\ell-4}\mathrm{d}s
	 =\int_{F_i}\partial_n (b_Tg_{\ell-3})q_{\ell-4}\mathrm{d}s
	 =\partial_n\lambda_i\int_{F_i} b_{F_i}g_{\ell-3}q_{\ell-4}\mathrm{d}s
	$$
	for $i=1,2,3$.
	Therefore,  $g_{\ell-3}|_{F_i}$ is a multiple of the $(\ell-3)$rd orthogonal
	polynomial $L_{\ell-3}^i$ over $F_i$ with respect to the even weight function
	$b_{F_i}$. That is, there is a coefficient $c_i$ such that
	$g_{\ell-3}|_{F_i} = c_i L_{\ell-3}^i$.
	Since $\ell-3$ is odd, it is known \cite{Spectral} that $L_{\ell-3}^i$
	is an odd function with
	$L_{\ell-3}^i(p_{i-1})= -L_{\ell-3}^i(p_{i+1})\neq 0$. 
    Since $g_{\ell-3}$ is continuous, evaluation of at the three vertices $p_1,p_2,p_3$
    results in
    $$
    c_1=-c_2, \quad c_2=-c_3, \quad c_3=-c_1.
    $$
    Therefore $c_1=c_2=c_3=0$ and so $g_{\ell-3}$ vanishes on the boundary
    $\partial T$.
    In consequence,
    $\partial_nv_h|_{\partial T}=0$.
    This proves the assertion in the case case $\ell=4$.
	If $\ell\geq 6$, we see that there exists a polynomial
	$g_{\ell-6}$ such that
	$$
	v_h = b_T^2 g_{\ell-6}.
	$$
    Since, by the vanishing degrees of freedom \eqref{e:dof2d_volume},
    the integral
    $\int_T v_h g_{\ell-6}\mathrm dx$ equals zero,
    we conclude that $g_{\ell-6}=0$ whence $v_h$ vanishes identically.
\end{proof}

\begin{remark}
 For the case of odd $\ell$, we illustrate that the above set of 
 degrees of freedom is generally not unisolvent.
 By $\hat L_{\ell-3}$ we denote the $(\ell-3)rd$ orthogonal polynomial
 over the reference interval $[-1,1]$ with respect to the weight function
 $x^2-1$.
 Similar as the construction in even order Crouzeix--Raviart elements,
 let $B:= b_T(\sum_{i=1}^3 \hat L_{\ell-3}(1-2\lambda_i)-\hat L_{\ell-3}(1))$.
 Obviously, $B$ vanishes on degrees of freedom \eqref{e:dof2d_nodal_edge}.
 A direct computation reveals that
 $$
 \partial_nB|_{e_i}= (\partial_n\lambda_i) b_{e_i}(\hat L_{\ell-3}(1-2\lambda_{i-1})+\hat L_{\ell-3}(1-2\lambda_{i+1})).
 $$ 
 Since $\ell-3$ is even, we have
 $\hat L_{\ell-3}(1-2\lambda_{i-1})|_{e_i}=\hat L_{\ell-3}(1-2\lambda_{i+1})|_{e_i}$,
 which implies that $B$ is a non-zero function whose normal derivative
 has vanishing edge moment up to the order $\ell-4$.
\end{remark}

Lemma \ref{uni2D} states that the functionals \eqref{e:dof2d_nodal_edge},
\eqref{e:dof2d_normalA}, \eqref{e:dof2d_volume} are linear independent
over $P_\ell(T)$
if $\ell$ is even.
In order to include the remaining three degrees of freedom
\eqref{e:dof2d_normalB}, we enlarge the space 
by three additional basis functions.
For any $i=1,2,3$ we set
$$
 \psi_i:=\lambda_{i+1}^{\ell-3}+\sum_{k=0}^{\ell-4}c_k\lambda_{i+1}^k \in P_{\ell-3}(T)
$$ 
where the coefficients $c_k$ are chosen such that such that $\psi_i$ satisfies
\begin{equation} 
  \int_{F_i}b^2_{F_i}\psi_iq_{\ell-4}\mathrm{d}s=0,
  \text{ for all } q_{\ell-4}\in P_{\ell-4}(F_i).
  \label{2Dbubble'}
\end{equation}
We then define bubble functions $\psi^B_i$, $i=1,2,3$ by
\begin{equation} 
\psi_i^B := b_Tb_{F_{i}}\psi_i
\label{2Dbubble}
\end{equation}
and define the shape function space as
\begin{equation}
	\mathcal{P}_{T,\ell} = P_{\ell}(T) + \mathrm{span}\{\psi_i^B : i=1,2,3\}.
\end{equation}  

\begin{mylem}
 If, $\ell\geq4$ is even,
  the degrees of freedom spanned by the functionals
  \eqref{e:dof2d_nodal_edge}--\eqref{e:dof2d_volume} are unisolvent for the space
  $\mathcal{P}_{T,\ell}$.
  \label{uni2Dgen}
\end{mylem}
\begin{proof}
Let $v_h\in \mathcal{P}_{T,\ell}$ vanish on degrees of freedom  \eqref{e:dof2d_nodal_edge}--\eqref{e:dof2d_volume}. We will show that $v_h\equiv 0.$
By definition of $\mathcal{P}_{T,\ell}$,
the function $v_h$ can be decomposed as
$$
 v_h= q_\ell +  \sum_{k=1}^{3}a_k\psi^B_k
$$
with some $q_\ell\in P_\ell(T)$ and real coefficients $a_1,a_2,a_3$.
Since the functions $\psi_k^B$ are zero on degrees of freedom \eqref{e:dof2d_nodal_edge},
	\eqref{e:dof2d_normalA},
	so must be the function $q_\ell$.
By Lemma~\ref{uni2D}, thus there exists a polynomial $q_{\ell-6}\in P_{\ell-6}(T)$
(zero if $\ell<6$) such that $q_\ell=b_T^2 q_{\ell-6}$.

Since the normal derivative of $q_\ell$ is zero over $\partial T$,
the vanishing of degrees of freedom \eqref{e:dof2d_normalB}
and elementary computations imply that
$$
  0 = \fint_{F_i} \lambda_{i+1}^{\ell-3}
       \partial_n \left(\sum_{k=1}^{3}a_k\psi^B_k\right)
      \mathrm{d}s
    =
    a_i (\partial_n\lambda_{i}) \fint_{F_i}\lambda_{i+1}^{\ell-3}b_{F_i}^2\psi_i\mathrm{d}s
    \quad\text{for }i=1,2,3.
$$
It follows from the definition of $\psi_i$ that the integral on the right-hand
side is nonzero, whence $a_i=0$ for $i=1,2,3$.
Therefore we have $v_h=q_\ell=b_T^2 q_{\ell-6}$.
The fact that $v_h$ vanishes on degrees of freedom \eqref{e:dof2d_volume}
finally implies that $q_{\ell-6}=0$ and thus $v_h=0$.
\end{proof}

We define the space of functions over $\Omega$ whose restriction to
any $T\in\mathcal T$ belongs to $\mathcal{P}_{T,\ell}$
by $\mathcal{P}_{\mathcal T,\ell}$
The global finite element space $V^{\ell}_{\mathcal T}$ defined as 
$$
\begin{aligned}
	V^{(\ell)}_{\mathcal T} := 
	H^1_0(\Omega)\cap 
	\{&
	v\in \mathcal{P}_{\mathcal T,\ell}: \nabla v \text{ is continuous at all vertices and}
	\\
	&
	\int_F [\partial_n v]q_{\ell-3}\,ds = 0
	\text{ for all interior faces } F \text{ and all } q_{\ell-3}\in P_{\ell-3}(F)
	\}
	.
\end{aligned}
$$

The crucial property of functions $v_h\in V^{(\ell)}_{\mathcal T}$ is
that the normal derivative of $v_h$ has some extra continuity in terms
of moments of degree one higher than required by the definition.
This implies that the finite element 
$(\Sigma_{T,\ell},\mathcal{P}_{T,\ell},\mathcal{T})$
satisfies the strong discrete Miranda--Talenti property \eqref{dMT}.

\begin{mylem}
	Let $v_h\in V^{(\ell)}_{\mathcal T}$ with $\ell\geq 4$ even.
	For any interior edge $F$, $v_h$ satisfies
	\begin{equation*}
		\int_F [\partial_nv_h]q_{\ell-2}\mathrm{d}s = 0
		\quad\text{for all } q_{\ell-2}\in P_{\ell-2}(F).
	\end{equation*}
	\label{2Dsuper}  
\end{mylem}
\begin{proof} 
    Let $F$ be an interior edge shared by elements $T_+$, $T_-$.
    In view of the degrees of freedom,
    it suffices to prove that
    $$
		\int_F [\partial_nv_h]\varphi_F\mathrm{d}s = 0
    $$
    for some $\varphi_F\in P_{\ell-2}(F)$ that is linear independent
    of all elements from $P_{\ell-3}(F)$.
    To this end, we
    choose $\varphi_F\in P_{\ell-2}(F)$ to be some nonzero function satisfying
	\begin{equation}
    \int_{F}b_F\varphi_Fq_{\ell-3}\mathrm{d}s=0
    \quad\text{for all } q_{\ell-3}\in P_{\ell-3}(F)
    .
    \label{temp1}	
    \end{equation}
	Any function $v_h\in V_{\mathcal T}$ can be expressed as 
	$v_h|_{T_\pm} = p_{\ell,\pm}+\psi_{\pm}^B$,
	where $p_{\ell,\pm}\in P_{\ell}(T_\pm)$ 
	and 
	$\psi_{\pm}^B\in \mathrm{span}\{\psi^B_i(T_\pm),i=1,2,3\}. $ 
	 Then
	\begin{align}\label{e:super_split}
		\int_{F}[\partial_n v_h]\varphi_F\mathrm{d}s
		=\int_{F}(\partial_np_{\ell,+}-\partial_np_{\ell,-})\varphi_F\mathrm{d}s+\int_{F}(\partial_n\psi^B_{\ell,+}-\partial_n\psi^B_{\ell,-})\varphi_F\mathrm{d}s.
	\end{align}
	We consider the first term on the right-hand side.
	Since $\nabla v_h$ continuous on all vertices and $\nabla\psi^B_{\pm}$ vanishes
	there,
	there exists some $q_{\ell-3}\in P_{\ell-3}(F)$ such that
	$(\partial_np_{\ell,+}-\partial_np_{\ell,-})$
	equals $b_Fq_{\ell-3}$. Relation \eqref{temp1} thus implies
	\begin{align*}
		\int_{F}(\partial_np_{\ell,+}-\partial_np_{\ell,-})\varphi_F\mathrm{d}s =
		\int_{F}b_Fq_{\ell-3}\varphi_F\mathrm{d}s = 0.
	\end{align*}
	For the second term of \eqref{e:super_split},
	by the definition of $\psi_i^B$, there exists a function $\psi_F\in P_{\ell-3}(F)$, such that
	$$
	\int_{F}(\partial_n\psi^B_{+}-\partial_n\psi^B_{-})\varphi_F\mathrm{d}s
	=\int_F b_F^2\psi_F\varphi_F\mathrm{d}s.
	$$
	According to \eqref{2Dbubble'}, ${\psi}_F$ satisfies
	$$
	\int_F b_F^2 \psi_F q_{\ell-4}\mathrm{d}s=0
	\quad\text{for all } q_{\ell-4}\in P_{\ell-4}(F).
	$$
	This shows that $\psi_F$ is an odd function.
	On the other hand $\varphi_F$ satisfies 
    \eqref{temp1} and therefore $\varphi_F$ is an even function.
    Thus,
	$$
	\int_F b_F^2 \psi_F \varphi_F\mathrm{d}s 
	=
	0.
	$$
	Hence, the right-hand side of \eqref{e:super_split} is zero
	and the assertion is proven.
\end{proof}

\begin{corollary}
 The spaces $V_{\mathcal T}^{(\ell)}$, with $\ell\geq 4$ even, satisfy the discrete
 Miranda--Talenti property from Definition~\ref{def:dmt}.
\end{corollary}
\begin{proof}
 Any $v_h\in V_{\mathcal T}^{(\ell)} $ satisfies $\partial^2_{t_F}v_h\in P_{\ell-2}(F)$. 
 Combining Lemma~\ref{DMT1} and Lemma~\ref{2Dsuper} 
 thus yields the proof.  
\end{proof}

\subsection{Example for $\ell=4$}
As an illustration, we give details on the 
the new finite element from Section~\ref{section2D}
for the case $\ell=4$.
On a triangle $T$, the shape function space reads
$$
\mathcal{P}_{T,4} = P_4(T)+\mathrm{span}\{b_Tb_{F_i}(\lambda_i-1/2):i=1,2,3\}.
$$
The degrees of freedom are the evaluation of the 
function and its gradient at the three vertices, the averages of the 
function over the three edges, and the zeroth and first order moments
of the normal derivative over the three edges.
An illustration is displayed in Figure~\ref{fig:2d_illustration}.
The global finite element space consists of all globally continuous
functions with homogeneous Dirichlet boundary values
such that the gradient is continuous in all vertices and the 
normal jumps over all interior edges have vanishing zeroth and 
first-order moments.
Lemma~\ref{2Dsuper} states that, additionally, the normal jumps
of such functions across interior edges automatically vanish
when integrated against quadratic polynomials.
One implication of this property is, as mentioned earlier, the validity
of the discrete Miranda--Talenti property.

A conforming finite element of the same order is the Bell element \cite{Bell1969A},
which also has 18 degrees of freedom and has third order convergence
and is thus comparable to this element with $\ell =4$.
A practical advantage of our nonconforming method is that the degrees of freedom
only involve first-order derivatives as degrees of freedom at the vertices,
which simplifies the implementation of boundary conditions.

\begin{figure}
 \begin{tikzpicture}[scale=2.2]
 \draw (0,0)--(1,0)--(.5,.8)--cycle;
 \foreach \x/\y  in {0/0,1/0,.5/.8}
      {  \fill (\x,\y) circle (1pt);
         \draw (\x,\y) circle (2pt);}
 \foreach \a/\b/\c/\d  in {.35/0/.35/-.2,
                           .65/0/.65/-.2,  
                     .825/.28/.975/.38,
                     .675/.52/.825/.62,
                     .175/.28/.025/.38,
                     .325/.52/.175/.62}
      {\draw[->,thick] (\a,\b)--(\c,\d);}
 \draw (.17,.05)--(.83,.05)
       (.9,.06)--(.55,.62)
       (.1,.06)--(.45,.62);
\end{tikzpicture}
\qquad\qquad
\begin{tikzpicture}[scale=2.2]
 \draw (0,0)--(1,0)--(.5,.8)--cycle;
 \foreach \x/\y  in {0/0,1/0,.5/.8}
      {  \fill (\x,\y) circle (1pt);
         \draw (\x,\y) circle (2pt);}
 \foreach \a/\b/\c/\d  in {.5/0/.5/-.2,
                     .75/.4/.9/.5,
                      .25/.4/.1/.5}
      {\draw[->,thick] (\a,\b)--(\c,\d);}
 \fill (.5,.2667) circle (1pt);
\end{tikzpicture}
\caption{Mnemonic diagram of the two-dimensional finite element for $\ell=4$
         (left) and $\ell=3$ (right).
         \label{fig:2d_illustration}}
\end{figure}

\subsection{Example for $\ell=3$}

The finite element from Section~\ref{section2D} is defined for even values
of $\ell\geq 4$.
In the low-order case, we can define an analogous element for the odd value
$\ell=3$.
It can be viewed as an alternative to the second-order Specht element.
Our element offers a simpler shape function space,
at the expense of one additional interior degree of freedom.
The construction is analogous to that of Section~\ref{section2D}.
The shape function space reads
\begin{equation*}
	\mathcal{P}_{T,3} = P_{3}(T)+\mathrm{span}\{b_Tb_{F_i}:i=1,2,3\}.
\end{equation*}
The degrees of freedom
(illustrated in Figure~\ref{fig:2d_illustration}) are the evaluation of the 
function and its gradient at the three vertices, the averages of the 
normal derivative over the three edges, and the volume average.
We denote this set of linear functionals as by $\Sigma_{T,3}$.

\begin{mylem}
	The shape function space $\mathcal{P}_{T,3}$ is unisolvent 
    by the degrees of freedom  $	\Sigma_{T,3}$.
\end{mylem}
\begin{proof}
   The functions
	Since $b_Tb_{F_i}$ and their gradients vanish at the vertices of 
    the triangle $T$.
    An argument along the lines of Lemma~\ref{uni2Dgen} therefore shows 
    that any function $v_h$ from $\mathcal{P}_{T,3}$ vanishing at all degrees 
    of freedom is the sum of a multiple of the volume bubble $b_T$
    and a linear combination of the functions $b_Tb_{F_i}$.
    Therefore there are real coefficients $c,a_1,a_2,a_3$ such that
    $$
     v_h = c\, b_T + \sum_{i=1}^3 a_i b_Tb_{F_i} .
    $$
    Since the degrees of freedom related to the normal derivative 
    vanish, an elementary computation reveals that
    $a_1=a_2=a_3$ and 
    $a_i=(\fint_{F_i}b_{F_i}\mathrm{d}s/\fint_{F_i}b^2_{F_i}\mathrm{d}s) c=-5c$.
    On the other hand, the volume average of $b_T$ equals $1/60$
    and the volume average of $b_Tb_{F_i}$ equals $1/630$.
    The constraint $\fint_T v_h\mathrm{d}x=0$ therefore implies
    $c=0$ and thus $v_h=0$.    
\end{proof}

The global finite element space for the the case $\ell=3$ is defined as
\begin{align*}
	V^{(3)}_{\mathcal T}=
	 H^1_0(\Omega)\cap 
	 \{ v \in \mathcal{P}_{\mathcal T,3} :
	 & \nabla v \text{ is continuous 
	 at all vertices and} &\\
	&\int_F [\partial_{n} v] \mathrm{d}s=0\text{ for any interior edge } F
	\}.
\end{align*}
Here, as in prior sections, the space $\mathcal{P}_{\mathcal T,3}$ consists of
functions that piecewise belong to $\mathcal{P}_{T,3}$.
A proof analogous to that of Lemma~\ref{2Dsuper}
shows that the finite element spaces enjoys a higher-order normal
continuity.
\begin{mylem}
		Any function $v_h\in V^{(3)}_{\mathcal T}$ satisfies for all interior edges $F$
		that
	\begin{equation*}
		\int_F [\partial_nv_h]q_1\mathrm{d}s = 0
		\quad\text{for all } q_1\in P_1(F).
	\end{equation*}  

\end{mylem}
\begin{proof}
	By the definition of the shape function space
	$\mathcal{P}_{T,3}$,
	there exist constants $\alpha, \beta$ such that 
	$[\partial_nv_h]|_{F}=\alpha b_e(1+\beta b_e)$.
	Since $\int_{F}[\partial_nv_h]\mathrm{d}s=0$ by the continuity conditions
	on $V_{\mathcal T}^{(3)}$, we have $0=\fint_F b_e(1+\beta b_e)\mathrm{d}s = 1/6+\beta/30$.
	This implies $\beta = -5$ and we
	directly compute 
	$\fint_{F}\lambda_e[\partial_nv_h]\mathrm{d}s=\alpha\fint_{F}\lambda_eb_e(1-5b_e)\mathrm{d}s =0. $
\end{proof}

By Lemma~\ref{DMT1}, the space $V^{(3)}_{\mathcal T}$ then satisfies the discrete
Miranda--Talenti property.

\section{Design of the finite element on three dimensions}\label{s:3d}
The construction of piecewise polynomial finite element functions 
in three dimensions
with $C^1$ continuity across edges requires the use of polynomials
of degree at least 5.
Analogous to the two-dimensional case, our shape functions
consists of two parts,
namely functions from the space $P_\ell(T)$ for polynomials of degree
$\ell\geq 5$ with respect to a simplex $T$,
plus carefully chosen
particular $H_0^1(T)$ bubble functions.
In view of the $C^1$ continuity at the edges,
the degrees of freedom related to the normal derivatives on faces
are the moments of order up to $\ell-4$, which is less than in
the two-dimensional case.

\subsection{Notation}
We begin by introducing some notation.
Throughout this section, $\mathcal T$ is a regular simplicial partition
of the open, bounded, connected 
polytopal Lipschitz domain $\Omega\subseteq\mathbb R^3$.
We denote by $\mathcal{N}, \mathcal{E},\mathcal{F}$ the sets
of vertices, edges, faces, respectively. 
Given a simplex $T\in\mathcal T$, the sets
$\mathcal{N}(T), \mathcal{E}(T),\mathcal{F}(T)$ are the sets of vertices,
edges, and faces of $T$.
We assign every edge $e\in\mathcal E$ with two fixed 
linear independent unit normals $n_1,n_2$.
On a simplex $T$ with vertices $a_1,a_2,a_3,a_4$,
the four barycentric coordinates are denoted by
$\lambda_1,\lambda_2,\lambda_3,\lambda_4$.
Given an edge $e$ with endpoints $a_i$, $a_j$,
the quadratic edge bubble function $b_e$ is defined as
$b_e:=\lambda_i\lambda_j$.
Given a face $F$ with vertices $a_i$, $a_j$, $a_k$,
the cubic face bubble function $b_F$ is defined as
$b_F:=\lambda_i\lambda_j\lambda_k$.
The quartic volume bubble is defined as 
$b_T:=\lambda_1\lambda_2\lambda_3\lambda_4$.

\subsection{Degrees of freedom}
Given any simplex $T\in\mathcal T$,
we define degrees of freedom (linear functionals) 
mapping any $v\in C^\infty(\mathbb R^3)$ to
\begin{subequations}
	\begin{align}
		\label{e:dof3d_nodal}
		v(p),~\nabla v(p), 
		D^2 v(p)&\quad \text{ for any }p\in \mathcal{N}(T)\\
			\label{e:dof3d_edge}
		\fint_{e}v q_{\ell-6}\mathrm{d}s~
		&\quad\text{ for any }q_{\ell-6}\in P_{\ell-6}(e),~e\in\mathcal{E}(T)\\
		\label{e:dof3d_edge_normal}
			\fint_{e}\partial_{n_j}v q_{\ell-5}\mathrm{d}s
			&\quad\text{ for any } q_{\ell-5}\in P_{\ell-5}(e),~e\in\mathcal{E}(T),~ j =1,2\\
		\label{F:dof3d_normalA}
		\fint_{F}\partial_nv q_{\ell-4}\mathrm{d}S 
        &\quad
		\text{ for any }q_{\ell-4}\in P_{\ell-4}(F),~F\in \mathcal{F}(T)\\
		\label{F:dof3d_normalB}
		\fint_{F}v q_{\ell-6}\mathrm{d}S 
		&\quad
		\text{ for any }q_{\ell-6}\in P_{\ell-6}(F),~F\in \mathcal{F}(T)
		\\
		\label{e:dof3d_volume}
		\fint_T v q_{\ell-4} \mathrm{d}x , 
		&\quad 
		\text{ for any }q_{\ell-4}\in P_{\ell-4}(T)
		.
	\end{align}
\end{subequations}
We use the convention that \eqref{e:dof3d_edge} and \eqref{F:dof3d_normalB}
are void if $\ell=5$.

\subsection{Shape functions}

	From the definition of degrees of freedom \eqref{e:dof3d_nodal}--\eqref{e:dof3d_volume}, 
	the normal derivative has $(\ell-4)$th order moment continuity
	in the sense of \S\ref{ss:meshes}.
	Our selection for of $H_0^1(T)$ bubble functions are functions
	orthogonal to particular $(\ell-2)$nd order polynomials 
	$\varphi_1,\dots,\varphi_{2\ell-3}$ on each face. 
	The latter polynomials are defined as follows.
Let $\ell\geq 5$ be a fixed integer and let $F\in\mathcal F(T)$ be a face
of the simplex $T\in\mathcal T$.
We abbreviate
$$
N_{\ell-4}:=\dim P_{\ell-4}(F).
$$
Let 
$(\phi_1,\dots,\phi_{N_{\ell-4}})$
be a given basis of $P_{\ell-4}(F)$, which
we extend to a basis of $P_{\ell-2}(F)$ 
by adding suitable polynomials $\varphi_1,\dots,\varphi_{2\ell-3}$
of degree $\ell-2$ satisfying the orthogonality relation
\begin{equation}
\int_{F}b_{F} q_{\ell-4} \varphi_{i}\mathrm{d}S=0
\quad\text{for all }q_{\ell-4}\in P_{\ell-4}(F)
\text{ and all } i=1,\dots,2\ell-3.
\label{express1}
\end{equation}
Next we will define $N_{\ell-4}$ bubble functions that
vanish on $\partial T$.
Let $X_F$ denote the subspace of $P_{\ell-2}(T)$ spanned
by the Lagrange basis functions that do not vanish identically
on $F$. The dimension of $X_F$ equals that of $P_{\ell-2}(F)$.
Let $\tilde\phi_1,\dots,\tilde\phi_{N_{\ell-4}} \in X_F$ satisfy
\begin{equation}\label{e:tildephi_orth}
\fint_{F}b_{F}^2\tilde{\phi}_i\phi_{j}\mathrm{d}S = \delta_{i,j} 
\text{ and }~\fint_{F}b_{F}^2\tilde{\phi}_i\varphi_{k}\mathrm{d}S = 0
\quad\text{for all }
\begin{cases}
 i,j=1,\dots,N_{\ell-4}\\k=1,\dots,2\ell-3.
\end{cases}
\end{equation}
Note that we have chosen the space $X_F$ for the purpose of uniqueness
of the functions $\tilde\phi_i$.
For the functions $\varphi_1,\dots,\varphi_{2\ell-3}$ from \eqref{express1}
we have
\begin{equation}
\begin{aligned}
	\int_{F}\partial_n b_Tb_{F}\tilde{\phi}_i\varphi_k\mathrm{d}S 
	= 
	(\partial_n\lambda_F)\int_{F}b^2_{F} \tilde{\phi}_i
\varphi_k 
	\mathrm{d}S
=	
0
\quad\text{ for all }
\begin{cases}
 i=1,\dots,N_{\ell-4}\\k=1,\dots,2\ell-3.
\end{cases}
	\label{express2}
\end{aligned}
\end{equation}
Therefore, the normal derivative  of $	b_Tb_F\tilde{\phi}_i$
$i=1,\cdots,N_{\ell-4}$, is orthogonal to $\varphi_k,k=1,\cdots,2\ell-3$. 
Given any $F\in \mathcal{F}(T)$, define the $N_{\ell-4}$ bubble functions as
\begin{equation}
	\xi_i^F :=  
	b_Tb_F\tilde{\phi}_i
	   -\tilde q_i,
	   \quad i=1,\cdots, N_{\ell-4}.
\label{express3}
\end{equation}
Here $\tilde q_1,\dots,\tilde q_{N_{\ell-4}}\in b^2_T P_{\ell-4}(T)$  satisfy
$$
\int_T\tilde q_i q_{\ell-4}\mathrm{d}x=
\int_T b_Tb_F\tilde{\phi}_i
q_{\ell-4}\mathrm{d}x
\quad\text{for all } q_{\ell-4}\in P_{\ell-4}(T).
$$

\begin{remark}\label{rem:xidof}
According to \eqref{express2} and \eqref{express3},
$\xi_i^F$ span the basis functions 
 dual to the degrees of freedom 
\eqref{F:dof3d_normalA}.
\end{remark}

We define the shape function space of our finite element as 
\begin{equation*}
\mathcal{P}_{T,\ell} = P_{\ell}(T) + \Phi_{T,\ell}^B
\end{equation*}
where
$\Phi_{T,\ell}^B 
= \mathrm{span}\{\xi_i^F: F\in\mathcal F(T),~
          i=1,\cdots,N_{\ell-4}\}.$

\begin{mylem}
    Let $\ell\geq 5$.
	The degrees of freedom defined by the functionals
	from \eqref{e:dof3d_nodal} -- \eqref{e:dof3d_volume} are 
	unisolvent for the shape function space
	$\mathcal{P}_{T,\ell}$.
\end{mylem}
\begin{proof}
To begin with, we observe that $\mathcal{P}_{T,\ell}=P_{\ell} + \Psi_{T,\ell}^B$
is a direct sum. To see this, we consider an arbitrary 
$w\in P_{\ell}(T)\cap\Psi_{T,\ell}^B$. 
Since $\Psi_{T,\ell}^B$ is spanned by the basis functions
$\xi_i^F$
and and each $\xi_i^F$ has the common factor $b_T$, it follows that $w$ can
be expressed as $w=b_Tq_{w}$ with some $q_{w}\in P_{\ell-4}(T)$. 
By construction, the volume integral over $T$ of each each $\xi_i^F$ against
polynomials of degree up to $(\ell-4)$ vanishes.
In particular,
$$
 0 = \int_T w q_w \mathrm{d}x
 = \int_T b_T q_w^2 \mathrm{d}x,
$$
which implies $q_w=0$ and therefore $w=0$.
This establishes that the direct sum property.
The dimension of \eqref{e:dof3d_nodal} -- \eqref{e:dof3d_volume} at most 
$\mathrm{dim}P_{\ell}(T)+4N_{\ell-4}=\mathrm{dim}\mathcal{P}_{T,\ell}$.
Let $v_h\in \mathcal{P}_{\ell}(T)$ vanish on the degrees of freedom
\eqref{e:dof3d_nodal} -- \eqref{e:dof3d_volume}.
We will show that $v_h \equiv 0$.
Since $v_h|_F\in P_\ell(F)$
for any face $F\in\mathcal{F}(T)$ and $v_h$ vanishes on
\eqref{e:dof3d_nodal},
\eqref{e:dof3d_edge},
\eqref{e:dof3d_edge_normal},
\eqref{F:dof3d_normalB},
the function $v_h$ can be written as   	
$$
  v_h = b_Tq_{\ell-4} + \sum_{i=1}^{N_{\ell-4}}\sum_{F\in \mathcal{F}(T)}c^F_i\xi_i^F
$$
for some $q_{\ell-4}\in P_{\ell-4}(T).$ 
By their definition \eqref{express3}, 
the functions $\xi_i^F$ vanish on \eqref{e:dof3d_volume}. Therefore, $v_h$ can be expressed as
$$
v_h = \sum_{i=1}^{N_{\ell-4}}\sum_{F\in \mathcal{F}(T)}c^F_i\xi_i^F.
$$
As $v_h$ vanishes on degrees of freedom \eqref{F:dof3d_normalA}
and, by Remark~\ref{rem:xidof}, the $\xi_i^F$ span the space corresponding
dual basis functions, we deduce
$c_i^F=0$ for all $i=1,\cdots,N_{\ell-4}$ and all $F\in \mathcal{F}(T)$.
This shows $v_h=0$ and concludes the proof.
\end{proof}

We denote by $\mathcal{P}_{\mathcal T,\ell}$ the space of functions whose restriction
to any $T\in\mathcal T$ belongs to $\mathcal{P}_{T,\ell}$.
The global finite element space reads
\begin{align*}
	V^{(\ell)}_{\mathcal T}=
	H^1_0(\Omega) \cap
	\left\{
	v\in \mathcal{P}_{\mathcal T,\ell}:
	\begin{aligned}
	 & v,~\nabla v, D^2 v~\text{are continuous at }\mathcal N,\\
	&v, \nabla v~\text{ are continuous at }\mathcal E,
	\text{ and }\int_F[\partial_nv]q \mathrm{d}S=0 \\&\text{ for all interior faces }F
	\text{ and }q\in P_{\ell-4}(F)
	\end{aligned}
	\right\}.
\end{align*}    

The subsequent lemma establishes improved normal continuity
for the elements of $V_{\mathcal T}^{(\ell)}$.
\begin{mylem}
 Any $v_h\in V^{(\ell)}_{\mathcal T}$, $\ell\geq 5$,
 satisfies for all interior faces
 $F$ that
\begin{equation*}
	\int_F [\partial_nv_h]p_{\ell-2}\mathrm{d}S = 0
	\quad\text{for all } p_{\ell-2}\in P_{\ell-2}(F).
\end{equation*}
\label{3Dsuper}
\end{mylem}
\begin{proof}
Let $F\in\mathcal F$ be an interior face shared by elements
$T_+$, $T_-$.
According to the definition \eqref{F:dof3d_normalA}, 
the asserted relation is satisfied for all test polynomials
of degree $\ell-4$ so that we only need to verify it
for polynomials $p_{\ell-2}\in P_{\ell-2}(F)$ that do not belong to 
$P_{\ell-4}(F)$.
We will take the linear independent
polynomials from \eqref{express1} as such test functions.
 Similar as in the 2D case, by the definition of
 $\mathcal P_{T,\ell}$,
 the function $v_h|_{T_\pm}$ can be rewritten 
 as $v_h|_{T_\pm}=p_{\ell,\pm} + \xi^B_\pm$, 
 where $p_{\ell,\pm}\in P_{\ell}(T_\pm)$ and
 $\xi^B_\pm\in \Phi^B_{T_\pm,\ell}.$
 Then we test with functions $\varphi_i$ defined in \eqref{express1}
 \begin{align}\label{e:3dproof_split} 
 \int_{F}[\partial_nv_h]\varphi_i\mathrm{d}S = \int_{F}(\partial_n p_{\ell,+}-\partial_n p_{\ell,-})\varphi_i\mathrm{d}S + \int_{F}(\partial_n\xi^B_+-\partial_n\xi^B_-)\varphi_i\mathrm{d}S.
 \end{align}
 The functions $\xi^B_\pm$ and $\nabla\xi^B_\pm$ vanish on all vertices and 
 edges.
  From the continuity properties of $ V^{(\ell)}_{\mathcal T}$,
 we see that the piecewise polynomial function $p_{\ell,\pm}$
 and its piecewise gradient are continuous at the vertices and edges.
 Therefore, the normal derivative jump on face $F$ vanishes on $\partial F$
 and thus there
 exists a polynomial $q_{\ell-4}\in P_{\ell-4}(F)$ such that
 $(\partial_np_{\ell,+}-\partial_np_{\ell,-})|_F = b_Fq_{\ell-4}$.
 The orthogonality relation \eqref{express1} shows that the first integral
 on the right-hand side of \eqref{e:3dproof_split} equals zero.
  In the remaining part of the proof we will show that the
 integrals
 $$
 \int_F\partial_n\xi^B_\pm\varphi_i\mathrm dS
 $$
 equal zero, which implies that the left-hand side of
 \eqref{e:3dproof_split} vanishes.
 Consider $\partial_n\xi^F_k$,
 $k=1,\cdots,N_{\ell-4}$.
 Since the volume bubble $\tilde{q}_j\in H_0^2(T)$ does not contribute
 to boundary integrals, we have from \eqref{e:tildephi_orth} that
\begin{align*}
	\int_{F\cap T_+} \partial_n\xi^F_j\varphi_i \mathrm{d}S
	& = (\partial_n\lambda_F)
	     \int_{F\cap T_+}b^2_F\tilde{\phi}_j
	       \varphi_i
	        \mathrm{d}S=0.
\end{align*}
This finishes the proof.
\end{proof}

As a consequence of Lemma~\ref{3Dsuper}
we note the following.
\begin{corollary}
 The spaces $V_{\mathcal T}^{(\ell)}$, with $\ell\geq 5$ even, satisfy the discrete
 Miranda--Talenti property from Definition~\ref{def:dmt}.
\end{corollary}
\begin{proof}
 The proof is consequence of Lemma~\ref{DMT1}
 and Lemma~\ref{3Dsuper}.
\end{proof}

\subsection{Example for $\ell=5$}
We give details for the lowest-order version of our three-dimensional
finite element.
On a simplex $T$, the shape function space reads
$$
\mathcal{P}_{T,5} = P_5(T)+\mathrm{span}\{\xi_i^F:F\in\mathcal F(T),i=1,2,3\}.
$$
The degrees of freedom are the point evaluation of the function,
the gradient, and the Hessian at the vertices;
the averages of the two normal derivatives along each edge;
the zeroth and first moments of the normal derivative over each face;
and the volume integrals of the function against polynomials of degree $\leq 1$.
The degrees of freedom are visualized in Figure~\ref{fig:3d_illustration}.
The local number of degrees of freedom therefore amounts to $68$.

\begin{figure}
\begin{tabular}{c|c|c|c}
4 vertices &
6 edges & 4 faces & volume
\\
\hline
\begin{tikzpicture}[scale=2.2]
\fill (0,0) circle (1pt);
\draw (0,0) circle (2.5pt);
\draw (0,0) circle (4.4pt);
\end{tikzpicture}
&
\begin{tikzpicture}[scale=1.8]
\draw[very thick] (0,0) -- (1,.6);
\draw[->,thick] (.5,.3)--(.7,.2);
\draw[->,thick] (.5,.3)--(.66,.5);
\end{tikzpicture}
&
\begin{tikzpicture}[scale=2.2]
\draw[fill=gray] (.6,.3) -- (1,.5) -- (.5,1)--cycle;
\draw[->,thick](.675,.525)--(.845,.695);
\draw[->,thick] (.775,.575)--(.975,.775);
\draw[->,thick] (.65,.7)--(.8,.85);
\end{tikzpicture}
&
\begin{tikzpicture}[scale=2.2]
\draw[] (.6,.3) -- (1,.5) -- (.5,1) -- cycle;
\draw[] (.6,.3) -- (.3,.5)--(.5,1);
\draw[dashed] (.3,.5)--(1,.5);
\fill (.5,.64) circle (1pt);
\fill (.58,.57) circle (1pt);
\fill (.5,.55) circle (1pt);
\fill (.59,.66) circle (1pt);
\end{tikzpicture}
\end{tabular}

\caption{Visualization of the degrees of freedom 
         of the three-dimensional finite element for $\ell=5$.
         \label{fig:3d_illustration}}
\end{figure}

\section{Application I: Non-divergence form equations}\label{s:nondiv}

As the first application of the finite elements satisfying a strong discrete 
Miranda--Talenti property \eqref{dMT} we present
equations in non-divergence form \eqref{nondiv}. 

\subsection{Cordes condition and strong solutions}
This section will list some basic results on strong solutions of 
problem \eqref{nondiv}. 
We assume that $A\in L^{\infty}(\Omega;\mathbb{R}^{d\times d})$, $d=2,3$,
is uniformly elliptic, i.e., there exist positive constants $\lambda$, $\Lambda$
such that
\begin{equation}
	\lambda \leq \eta^TA\eta\leq \Lambda\quad \text{a.e. in }\Omega
	\quad \text{for all }\eta\in \mathbb{R}^d \text{ with } \eta^T\eta = 1.
	\label{uniell}
\end{equation}
For the well-posedness of \eqref{nondiv} we impose an additional
condition known as Cordes condition \cite{MT,Suli2013}.
\begin{myDef}[Cordes condition]
	There exist an $\varepsilon\in (0,1]$ such that 
	\begin{equation}
		\frac{|A|^2}{(\mathrm{tr}A)^2}\leq \frac{1}{d-1+\varepsilon}.
		\label{cordes}
	\end{equation}
\end{myDef}
In the case $d=2$ the Cordes condition is a consequence of the 
uniform ellipticity \eqref{uniell} while it is an essential condition
if $d=3$.
Following \cite{Suli2013} we introduce the scaling function
$\gamma := |A|^{-2}\operatorname{tr}A$. Then the variational formulation of \eqref{nondiv} is given as  
\begin{equation}
	A(u,v) : = \int_{\Omega}\gamma (A:D^2u)\triangle v\mathrm{d}x = \int_{\Omega}\gamma f\triangle v\mathrm{d}x,\quad \text{for all } v\in V. 
	\label{nondivvar}
\end{equation} 

The Cordes condition guarantees adequate control on
the distance between $A:D^2 v$ and $\triangle v$.
The combination with the Miranda--Talenti inequality \eqref{MT} 
results in the following well-posedness result for \eqref{nondiv}.
\begin{mylem}[\cite{Suli2013}, Theorem 3]
	Let $A$ satisfy \eqref{uniell} and \eqref{cordes}.
     Then, for any open set $U\subset \Omega$ and $v\in H^2(U)$, we have
	\begin{equation}
		|\gamma A:D^2 v- \triangle v|\leq \sqrt{1-\varepsilon}|D^2 v|,\quad a.e. \text{ in } U,
		\label{cordes2}
	\end{equation}
	where the $\varepsilon$ is defined in \eqref{cordes}. 
	Moreover, for any $f\in L^2(\Omega)$ problem \eqref{nondivvar}
	has a unique solution $u\in H^2(\Omega)\cap H_0^1(\Omega)$ and satisfies
	\begin{equation}
		\|D^2 u\|_0 \leq
		\|\triangle u\|_0 \leq
		\frac{\|\gamma\|_{\infty}}{1-\sqrt{1-\varepsilon}}\|f\|_0.
	\end{equation}
\end{mylem}

\subsection{Finite element discretization}
The discrete problem corresponding to \eqref{nondivvar}
seeks 
$u_h\in V_{\mathcal T}\subseteq H^2(\mathcal{T})\cap H_0^1(\Omega)$, such that
\begin{equation}
	A_h(u_h,v_h):=\int_\Omega (\gamma A:D^2_{\mathcal T} u_h)\triangle_{\mathcal T} v_h\mathrm{d}x 
	= 
	\int_\Omega \gamma f \triangle_{\mathcal T} v_h \mathrm{d}x
	 \quad \text{for all } v_h\in V_{\mathcal T}.
	\label{dnondiv}	
\end{equation} 

Since the exact solution $u$ satisfies $A:D^2 u=f$ pointwise
almost everywhere in $\Omega$, the following analogue of what
is called Galerkin orthogonality in conforming methods is valid
\begin{equation}\label{e:galerkinorthogonality}
A_h(u-u_h,v_h)=0 \quad\text{for all }v_h\in V_{\mathcal T}. 
\end{equation}

The basic error estimate is as follows.

\begin{myTheo}
	Let $\Omega$ be a bounded convex polytopal domain with simplicial triangulation 
	$\mathcal{T}$ and let $A\in L^{\infty}(\Omega;\mathbb R^{d\times d})$
	satisfy uniform ellipticity \eqref{uniell} and the Cordes condition
	\eqref{cordes}.
	Let $u$ be the solution of \eqref{nondivvar}.
	If $V_{\mathcal T}\subseteq H^2(\mathcal{T})\cap H_0^1(\Omega) $ 
	satisfies the strong discrete Miranda--Talenti property, 
	then \eqref{dnondiv} has a unique solution $u_h\in V_{\mathcal T}$,
	which satisfies
    \begin{equation*}
		\|D^2_{\mathcal T}(u-u_h)\|_0
    \leq 
    \left(1+\frac{\Lambda}{1-\sqrt{1-\varepsilon}}\right)
    \inf_{w_h\in V_{\mathcal T}}\|D^2_{\mathcal T}(u-w_h)\|_0 . 
    \end{equation*}
	\label{duni}
\end{myTheo}

\begin{proof}
We abbreviate $\delta:=1-\sqrt{1-\varepsilon}$.
	Since $V_{\mathcal T}$ satisfies \eqref{dMT}, 
	$\|\triangle_{\mathcal T}\cdot\|_0$ defines a norm on $V_{\mathcal T}$.
	Combining \eqref{cordes2} and \eqref{dMT}, we have
	for any $v_h\in V_{\mathcal T}$ that
	\begin{align}\label{e:d_coerc}
	\begin{aligned}
		A_h(v_h,v_h) &= \|\triangle_{\mathcal T} v_h\|^2_0 
		- \sum_{T \in \mathcal{T}}\int_T(\gamma A:D^2v_h-\triangle v_h)\triangle v_h\mathrm{d}x
		\geq \delta\|\triangle_{\mathcal T} v_h\|^2_{0}.
    \end{aligned}
	\end{align}
	Since, by the discrete Miranda--Talenti inequality,
	$A_h(\cdot,\cdot)$ is also bounded on $V_{\mathcal T}$,
	the Lax--Milgram Theorem implies that \eqref{dnondiv} has unique 
	solution $u_h\in V_{\mathcal T}$. 
    Then for any $w_h\in V_{\mathcal T}$, the triangle inequality implies
\begin{equation}\label{e:errest_triang}
\|D^2_{\mathcal T}(u-u_h)\|_0\leq \|D^2_{\mathcal T}(u-w_h)\|_0 + \|D^2_{\mathcal T}(u_h-w_h)\|_0.
\end{equation}
For the analysis of the the second term on the right-hand side,
we abbreviate $e_h:=u_h-w_h$.
The discrete coercivity
\eqref{e:d_coerc} implies
$$
\|\triangle_{\mathcal T} e_h\|_0^2
\leq \delta^{-1}A_h(e_h,e_h)
= \delta^{-1} A_h(u_h-w_h,e_h).
$$
The orthogonality \eqref{e:galerkinorthogonality} and the Cauchy
inequality show that
$$
A_h(u_h-w_h,e_h) 
= A_h(u-w_h,e_h)
\leq \Lambda\|D^2_{\mathcal T}(u-w_h)\|_0 \|\triangle_{\mathcal T}e_h\|_0.
$$
We use the discrete Miranda--Talenti estimate \eqref{dMT} and
the combination of the two foregoing displayed estimates to
conclude
$$
\|D^2_{\mathcal T} e_h\|_0
\leq \|\triangle_{\mathcal T} e_h\|_0
\leq \delta^{-1}\Lambda\|D^2_{\mathcal T}(u-w_h)\|_0
.
$$
Since $w_h\in V_{\mathcal T}$ was arbitrary, the combination with
\eqref{e:errest_triang} proves the assertion.
\end{proof}

We apply Theorem~\ref{duni} to our families of finite elements
from the previous sections and obtain the following 
asymptotic error estimates.
\begin{myTheo}
	Let $\Omega\subseteq \mathbb R^d$ for $d=2,3$ be a bounded convex polygonal domain 
	triangulated by $\mathcal{T}$ and 
	let $A\in L^{\infty}(\Omega)$  satisfy conditions
	\eqref{uniell} and \eqref{cordes}.
	Let $u_h\in V^{(\ell)}_{\mathcal T}$ denote the discrete solution to
	\eqref{dnondiv}, where
	$$
	\begin{cases}
	 \ell=3 \text{ or } \ell \text{ is even with }\ell\geq 4 &\text{if } d=2,\\
	 \ell\geq 5 &\text{if }d=3 .
	 \end{cases}
	 $$
	If the solution $u$ to \eqref{nondivvar}
	satisfies $u\in H^s(\Omega)\cap H_0^1(\Omega)$ for some
	$0<s\leq\ell+1$, and $u_h$ is the solution of 
	 \eqref{dnondiv}, then
	\begin{equation*}
		\|D^2_{\mathcal T} (u-u_h)\|_{0}\leq Ch^{s-2}|u|_s
	\end{equation*}
	where $|\cdot|_s$ denotes the usual $H^s(\Omega)$ seminorm.
\end{myTheo}

\begin{proof}
For $d=2$ and for $d=3$ with $s\geq 7/2$, the result follows from
standard interpolation error estimates \cite{finiteBrenner}.
For the remaining case $d=3$ and $2<s<7/2$, a suitable quasi-interpolation
operator can be designed by averaging as outlined in \cite{finitewang}.
We skip the details, which are based on standard techniques.
\end{proof}

\section{Application II: Biharmonic equation}\label{s:bih}
The second application of our new finite element families is the biharmonic equation.
Let $\Omega\subset \mathbb{R}^d$, $d=2,3$ be a bounded,
open, polytopal Lipschitz domain.
The weak form of the biharmonic equation $\triangle^2 u =f$ for some
right-hand side $f\in L^2(\Omega)$ and clamped boundary conditions
$u=\partial_n u=0$ over $\partial\Omega$
seeks $u\in H^2_0(\Omega)$ such that
\begin{equation}
\int_\Omega D^2u:D^2v\mathrm{d}x=\int_\Omega fv\mathrm{d}x
\quad\text{for all }v\in H^2_0(\Omega).
\label{bihar}
\end{equation}
Given a space $V_{\mathcal T}$ of piecewise polynomial functions with respect to
a triangulation $\mathcal T$, the discrete problem seeks
$u_h\in V_{\mathcal T}$ such that
\begin{equation}
\int_\Omega D^2_{\mathcal T}u_h:D^2_{\mathcal T}v_h\mathrm{d}x=\int_\Omega fv_h\mathrm{d}x
\quad\text{for all }v_h\in V_{\mathcal T}.
\label{dbihar}
\end{equation}
Given a face $F\in\mathcal F$,
let the face patch $\omega_F$ be the union of the simplices sharing the face $F$.
Given a measurable subdomain $G\subseteq\overline\Omega$,
let $\mathrm{P}^{m}_G$ denote the $L^2$-orthogonal projection operator to the 
polynomials of degree not larger than $m$ over $G$.

\begin{remark}\label{rem:clampedBC}
 In order to apply subspaces of the spaces
 $V^{(\ell)}_{\mathcal T}$ from Sections~\ref{s:2d}--\ref{s:3d},
 to the problem with clamped boundary conditions, we additionally need to
 enforce a discrete analogue of the condition $\partial_n u=0$
 on $\partial\Omega$.
 To this end, all degrees of freedom related to the normal derivative
 on boundary faces and all edge degrees of freedom for
 boundary edges are set to zero.
 For $d=2$ all degrees of freedom related to the boundary vertices
 are set to zero. For $d=3$ the vertex degrees of freedom
 $\partial_{nn}^2 u(z)$ for any boundary vertex $z\in\mathcal N$
 is left free unless $z$ is a boundary point where $\partial\Omega$
 is not differentiable.
\end{remark}

By applying a result of \cite{Hu&Zhang2019}, we directly obtain an error 
estimate for our finite element families adapted to the biharmonic
problem.

\begin{myTheo}
Let $u\in H^2_0(\Omega)$ be the solution to \eqref{bihar}.
Let the integer $\ell$ satisfy
	$$
	\begin{cases}
	 \ell=3 \text{ or } \ell \text{ is even with }\ell\geq 4 &\text{if } d=2,\\
	 \ell\geq 5 &\text{if }d=3 
	 \end{cases}
	 $$
and let $V_{\mathcal T}$ denote the subspace of the space
$V^{(\ell)}_{\mathcal T}$ (from Sections~\ref{s:2d}--\ref{s:3d})
described in Remark~\ref{rem:clampedBC}.
Then problem \eqref{dbihar} has a unique solution $u_h\in V_{\mathcal T}$,
which satisfies
\begin{align*}
\|D^2_{\mathcal T}(u-u_h)\|_0\leq C&\bigg(\inf_{v_h\in V_{\mathcal T}}\|D^2_{\mathcal T}(u-v_h)\|_0 
 + \Big(\sum_{F\in\mathcal F}\|(1-\mathrm{P}^{\ell-2}_{\omega_F})D^2u\|^2_{0,\omega_F}\Big)^{1/2}\\
 \quad
& +  \Big(\sum_{T \in \mathcal{T}}h_T^4\|(1-\mathrm{P}^{\ell-3}_{T})f\|^2_{0,T}\Big)^{1/2}   
\bigg).
\end{align*} 
If in addition the solution satisfies $u\in H_0^2(\Omega)\cap H^s(\Omega)$
for some $0<s\leq \ell+1$, then
$$
\|D^2_{\mathcal T}(u-u_h)\|_0
\leq C h^{s-2}|u|_s
+
\Big(\sum_{T \in \mathcal{T}}h_T^4\|(1-\mathrm{P}^{\ell-3}_{T})f\|^2_{0,T}\Big)^{1/2}   
\bigg).
$$
\end{myTheo}
\begin{proof}
The definition of the spaces $V_{\mathcal T}$ ensures that
$\|D^2_{\mathcal T}\cdot\|_0$ defines a norm over $V_{\mathcal T}$
so that existence and uniqueness of discrete solutions follow from
standard arguments \cite{finiteBrenner}.
An adaptation of Lemma~\ref{2Dsuper} and Lemma~\ref{3Dsuper} 
shows for all faces $F$ that
$$
\int_F[\nabla v_h]q_{\ell-2}\mathrm{d}S = 0\quad 
\text{for all } q_{\ell-2}\in P_{\ell-2}(F).
$$ 
An application of Theorem~2.1 of \cite{Hu&Zhang2019} thus concludes the proof. 
\end{proof}

\section{Numerical Experiments}\label{s:num}
In this section, we present numerical results for the 
elliptic problem in nondivergence form and the biharmonic equation
(with clamped boundary conditions) in two and three space dimensions
on uniformly refined meshes.
We always display relative errors in the $L^2$ norm and the 
discrete $H^2$ seminorm,
abbreviated by $e_{L^2}^{\mathrm{rel}}$
and $e_{H^2}^{\mathrm{rel}}$, respectively.

\subsection{Nondivergence-form operator in two dimensions: first example}
We consider the elliptic operator in nondivergence form given through
the coefficient
$$
A(x,y) = \left(\begin{matrix}
	2&xy/|xy|\\
	xy/|xy|&2
\end{matrix}\right)
$$
on the square $\Omega=(-1,1)^2$ with an initial mesh that is
aligned to the discontinuities of $A$.
Then $\gamma =2/5,\varepsilon = 3/5.$
The exact solution is given by
$u(x,y) =xy(1-e^{(1-|x|)})(1-e^{1-|y|})$.
We compare three numerical methods:
the quadratic Specht triangle and our finite element space
$V_{\mathcal T}^{(\ell)}$ from Section~\ref{s:2d} with $\ell=3$ and $\ell=4$.
The results are displayed in Tables~\ref{tab:2dnondiv_a}--\ref{tab:2dnondiv_b}.
The quadratic Specht element and our element with $\ell=3$
converge with order close to 2 in the discrete $H^2$ seminorm
and with an order better than 3 in the $L^2$ norm.
For $\ell=4$, we observe third-order convergence in the discrete $H^2$ seminorm
and a convergence order better than 4 in $L^2$.

\begin{table}
	\begin{center} 
		\begin{tabular}{ccccccccc}\toprule
		&\multicolumn{4}{c}{quadratic Specht}& \multicolumn{4}{c}{$V_{\mathcal T}^{(3)}$}\\
 \cmidrule(lr){2-5}\cmidrule(lr){6-9}
	$h$             &  $e_{L^2}^{\mathrm{rel}}$&  rate & $e_{H^2}^{\mathrm{rel}}$
	               & rate & $e_{L^2}^{\mathrm{rel}}$ & rate &$e_{H^2}^{\mathrm{rel}}$ &  rate \\
			\hline
	1/16	       & \num{3.85e-03}&    &\num{2.73e-02}&   & \num{3.72e-03} &  & \num{2.93e-02}  &  \\
	1/32	      &\num{4.13e-04}&  3.22   &\num{7.60e-03}& 1.85   & \num{4.01e-04} & 3.21 & \num{8.13e-03}  &  1.85\\
	1/64	       &\num{4.16e-05}& 3.31   &\num{2.00e-03}&1.92   & \num{4.03e-05} & 3.32 & \num{2.16e-03}  & 1.94 \\
	1/128	       & \num{4.04e-06}& 3.36   &\num{5.17e-04}& 1.92  & \num{2.27e-06} & 4.15 & \num{5.54e-04}  &1.92  \\
	\bottomrule
		\end{tabular}
	\end{center} 
	\caption{Convergence history for the nondivergence-form operator in two dimensions:
    quadratic Specht element and $V_{\mathcal T}^{(\ell)}$ with $\ell=3$.
    \label{tab:2dnondiv_a}}
\end{table}

\begin{table}
	\begin{center} 
			\begin{tabular}{ccccc}
	\toprule
     $h$    &   $e_{L^2}^{\mathrm{rel}}$&rate& $e_{H^2}^{\mathrm{rel}}$&rate\\
     \midrule
     1/8	&	\num{5.17e-04}&		&\num{6.12e-02}&\\
     1/16	&	\num{2.73e-05}& 4.24&\num{8.20e-04}&2.88 \\
     1/32	&	\num{1.33e-06}& 4.36&\num{1.06e-04}&2.95\\
     1/64 	&	\num{4.69e-08}& 4.82& \num{1.35e-05}&2.98\\
	\bottomrule
	\end{tabular}
\end{center} 
\caption{Convergence history for the nondivergence-form operator in two dimensions:
    $V_{\mathcal T}^{(\ell)}$ with $\ell=4$.
    \label{tab:2dnondiv_b}}
\end{table}

\subsection{Nondivergence-form operator in two dimensions: second example}
For the second example we choose the same coefficient matrix $A$ and domain $\Omega$ as before.
Given a positive parameter $\iota>0$, we consider the exact solution
$$
u(x,y) = |(x,y)|^{3/2}\left(\frac{y+1}{2}+ \frac{e^{-1/\iota}-e^{(y-1)/2\iota}}{1-e^{-1/\iota}}\right)
$$
with boundary layer near the line $(-1,1)\times \{1\}$
and a point singularity at the origin
so that $u\in H^{5/2-\nu}(\Omega)$ for any $\nu>0$.
Qualitatively, this inclusion holds for any positive
choice of $\iota$, but small values of $\iota$ will
cause a large $H^2(\Omega)$ norm of $u$, which makes the
point singularity less visible on moderately fine meshes.
We apply our element with $\ell=3$ to this example.
We consider local mesh refinement steered by local contributions of 
the heuristic refinement indicator
\begin{equation*}
	\eta := \sqrt{  \sum_{T\in\mathcal T} \|f-A:D^2 u_h\|_{L^2(T)}^2 }.
\end{equation*}
 Figure \ref{MS} shows the results with $\iota \in \{1/100,1/10\}$.
 Due to the boundary layer and the singularity at the origin, 
 uniform mesh refinement results in suboptimal convergence rates.
 However, the adaptive method recovers the optimal convergence
 rate in terms of the total number of degrees of freedom. 
 We observe for $\iota=1/100$ a strong impact of the boundary layer
 such that the convergence rate is visible after $10^3$ degrees of freedom only.
 For $\iota=1/10$ this effect is much weaker.
 Figure~\ref{MS01} displays an adaptive mesh for $\iota=1/10$
 with about 5\,000 degrees of freedom.
 It can be observed, that the boundary layer is resolved and that the 
 refinement concentrates around the singularity.
 A further observation is that the quantity $\eta$ is equivalent to the 
 error in the $H^2$ norm.
\begin{figure}
	\includegraphics[width=.7\textwidth]{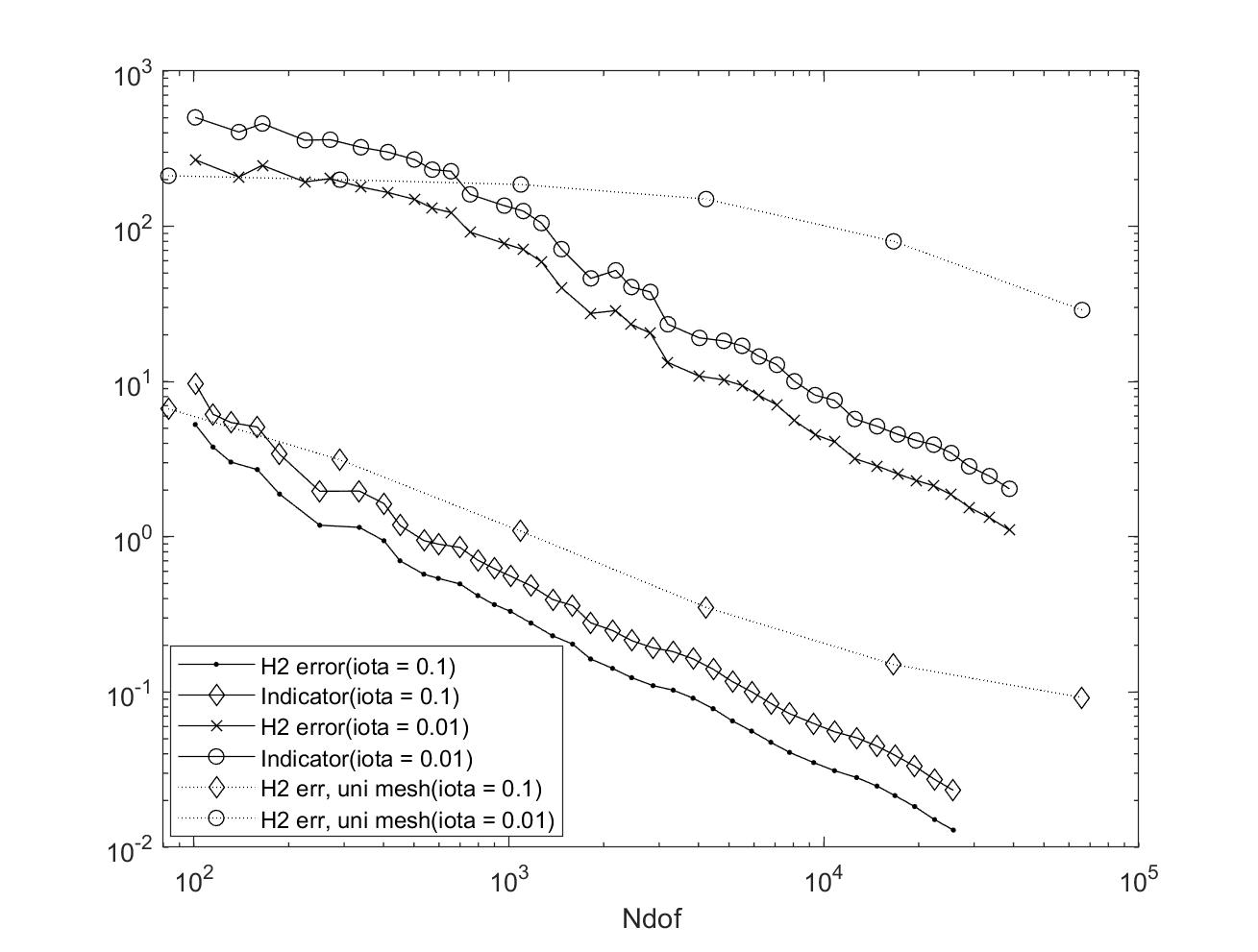}
	\caption{Values of the error and $\eta$ for our element with $\ell=3$ 
	    in the second example with parameter $\iota\in \{1/100,1/10\}$.}
\label{MS}
\end{figure}
\begin{figure}
	\includegraphics[width=.7\textwidth]{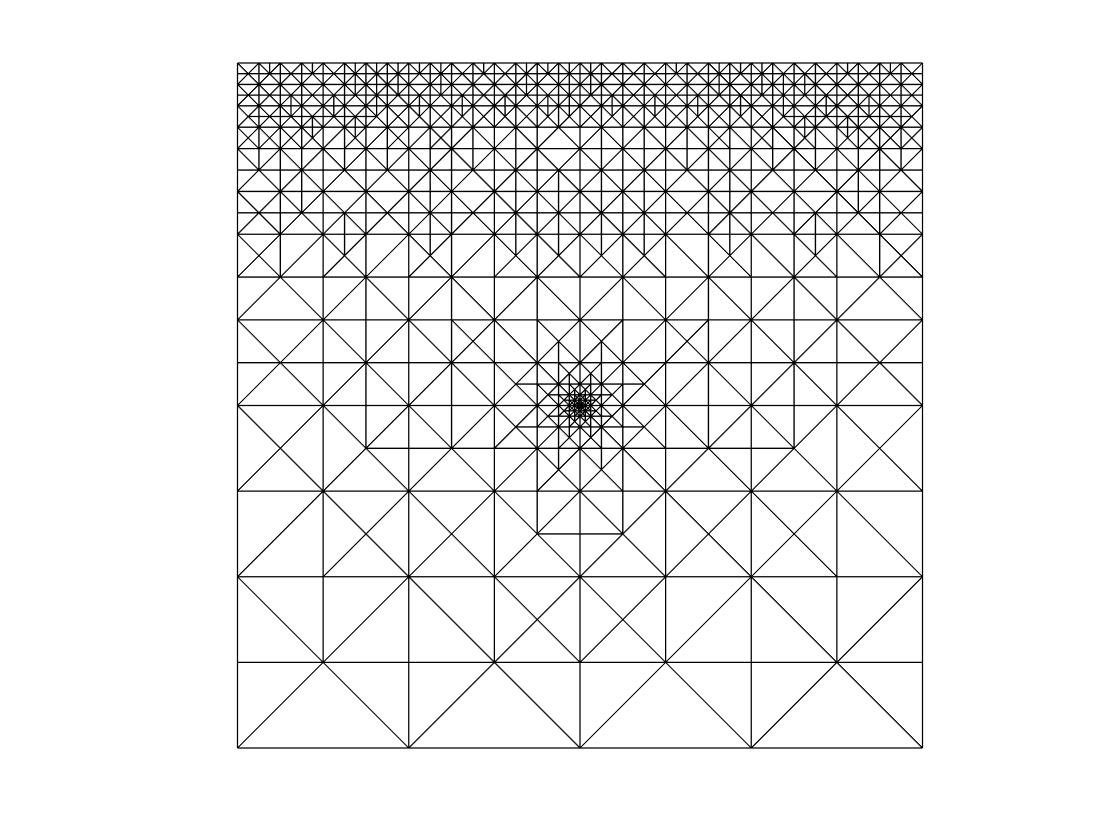}
	\caption{The mesh with around $5\,127$ degrees of freedom generated by our element with $\ell=3$ when $\iota = 1/10$.}
	\label{MS01}
\end{figure}
\subsection{Biharmonic operator in two dimensions}
We consider the biharmonic equation with clamped boundary conditions on
the square $\Omega=(0,1)^2$.
The exact solution is chosen as
$u(x,y)=(\mathrm{sin}(\pi x)\mathrm{sin}(\pi y))^2$.
We compare the three methods based 
the quadratic Specht triangle and our finite element space
$V_{\mathcal T}^{(\ell)}$ from Section~\ref{s:2d} with $\ell=3$ and $\ell=4$.
The results are displayed in Tables~\ref{tab:2dbih_a}--\ref{tab:2dbih_b}.
The quadratic Specht element and our element with $\ell=3$
converge with order close to 2 in the discrete $H^2$ seminorm
and with order close to 4 in the $L^2$ norm.
For $\ell=4$, we observe third-order convergence in the discrete $H^2$ seminorm
and a convergence order better than 5 in $L^2$.

\begin{table}
	\begin{center} 
		\begin{tabular}{ccccccccc}\toprule
			&	\multicolumn{4}{c}{quadratic Specht}& \multicolumn{4}{c}{$V_{\mathcal T}^{(3)}$}\\
			\cmidrule(lr){2-5}\cmidrule(lr){6-9}
			$h$             &  $e_{L^2}^{\mathrm{rel}}$&  rate & $e_{H^2}^{\mathrm{rel}}$ & rate & $e_{L^2}^{\mathrm{rel}}$ & rate &$e_{H^2}^{\mathrm{rel}}$ &  rate \\
			\hline
	1/16	       & \num{6.53e-03}&       &\num{7.78e-02}&       & \num{6.40e-03} &      & \num{5.21e-02}  &  \\
	1/32	       & \num{6.61e-04}&  3.32 &\num{2.52e-02}& 1.67  & \num{6.27e-04} & 3.34 & \num{1.63e-02}  &  1.68\\	
	1/64	       & \num{5.04e-05}& 3.71  &\num{6.71e-03}& 1.87  & \num{4.72e-05} & 3.73 & \num{4.41e-03}  & 1.88 \\
	1/128	       & \num{3.35e-06}& 3.91  &\num{1.72e-03}& 1.96  & \num{3.14e-06} & 3.91 & \num{1.13e-04}  &1.96   \\
	\bottomrule
		\end{tabular}
	\end{center} 
	\caption{Convergence history for the biharmonic equation in two dimensions:
    quadratic Specht element and $V_{\mathcal T}^{(\ell)}$ with $\ell=3$.
    \label{tab:2dbih_a}}
\end{table}

\begin{table}
	\begin{center} 
		\begin{tabular}{ccccc}\toprule
			$h$  & $e_{L^2}^{\mathrm{rel}}$&rate&$e_{H^2}^{\mathrm{rel}}$ &rate\\
			\midrule
			1/8	&\num{4.26e-03}&    &  \num{5.99e-02}& \\
			1/16&\num{1.09e-05}&5.27&\num{9.06e-03}&2.73 \\
			1/32&\num{2.35e-06}&5.54&\num{1.20e-03}&2.91\\
			1/64&\num{5.79e-08}&5.34& \num{1.53e-04}&2.97\\
			\bottomrule
		\end{tabular}
	\end{center} 
     \caption{Convergence history for the biharmonic equation in two dimensions:
    $V_{\mathcal T}^{(\ell)}$ with $\ell=4$.
    \label{tab:2dbih_b}}
\end{table}

\subsection{Nondivergence-form operator in three dimensions}
We consider the elliptic operator in nondivergence form given through
the coefficient
$$
A (x,y,z)= \left(\begin{matrix}
	3&xyz/|xyz|&xyz/|xyz|\\
	xyz/|xyz|&3&xyz/|xyz|\\
	xyz/|xyz|&xyz/|xyz|&3
\end{matrix}\right).
$$
The domain is $\Omega=(-1,1)^3$
and we compute 
$\gamma = 11/27,~\varepsilon = 5/11.$ 
We choose $u(x,y,z) = \mathrm{sin}(\pi x) \mathrm{sin}(\pi y) \mathrm{sin}(\pi z)$
as the exact solution.
We approximate the problem with our finite element space
$V_{\mathcal T}^{(\ell)}$ from Section~\ref{s:3d} with $\ell=5$.
The results are displayed in Table~\ref{tab:3dnondiv}.
We observe convergence of order 4 in the discrete $H^2$ seminorm
and convergence of order 6 in the $L^2$ norm.

\begin{table}
	\begin{center} 
		\begin{tabular}{ccccc}\toprule
		$h$  &$e_{L^2}^{\mathrm{rel}}$&rate&$e_{H^2}^{\mathrm{rel}}$&rate\\
		\midrule
		1/2  &\num{4.96e-01}& &  \num{5.41e-01}& \\
		1/4  &\num{1.51e-02}&5.04 &\num{8.68e-02}&2.64\\
		1/8  &\num{2.81e-04}&5.75&\num{8.97e-03}&3.28\\
		1/16 &\num{2.54e-06}&6.79& \num{5.20e-04}&4.11\\
		\bottomrule
		\end{tabular}
	\end{center} 
	\caption{Convergence history for the nondivergence-form operator in three dimensions
	         with the lowest-order method $(\ell=5)$.
	      \label{tab:3dnondiv}}
\end{table}

\subsection{Biharmonic operator in three dimensions}
We consider the biharmonic equation with clamped boundary conditions on
the cube $\Omega=(0,1)^3$.
The exact solution is given by
$u(x,y,z) = (\mathrm{sin}(\pi x)\mathrm{sin}(\pi y)\mathrm{sin}(\pi z) )^2.$
We approximate the problem with our finite element space
$V_{\mathcal T}^{(\ell)}$ from Section~\ref{s:3d} with $\ell=5$.
The results are displayed in Table~\ref{tab:3dbih}.
We observe convergence of order 4 in the discrete $H^2$ seminorm
and convergence better than of order 6 in the $L^2$ norm.

\begin{table}
	\begin{center} 
		\begin{tabular}{ccccc}\toprule
			$h$    &$e_{L^2}^{\mathrm{rel}}$&rate&$e_{H^2}^{\mathrm{rel}}$&rate\\
			\midrule
			1/2& \num{6.14e-02}&     & \num{1.71e-01} &   \\
			1/4&\num{4.79e-03} & 3.68& \num{4.07e-02} &2.07\\
			1/8&\num{5.69e-05}&6.40  & \num{3.35e-03}&3.60\\
			1/16&\num{4.59e-07} &6.95& \num{1.78e-04} &4.23\\
			\bottomrule
		\end{tabular}
	\end{center} 
	\caption{Convergence history for the biharmonic equation in three dimensions
	         with the lowest-order method $(\ell=5)$.
	\label{tab:3dbih}}
\end{table}

\section{Conclusions}\label{s:conclusion}

Based on the property of $C^1$ continuity across $(d-2)$-dimensional subsimplices,
we have constructed continuous but $H^2$-nonconforming finite element spaces
that satisfy the Miranda-Talenti inequality in 2D and 3D.
This property avoids the requirement of additional stabilization terms 
that would usually be needed when $H^2$ nonconforming elements are 
applied to elliptic equations in non-divergence form under the
Cordes condition.
The finite element can be applied to other problems, 
such as Hamilton--Jacobi--Bellman equations, biharmonic equations,
or singularly perturbed fourth order elliptic problems.
We have implemented the new method for a biharmonic and a nondivergence-form
model problem.
In all experiments, we observe the optimal convergence rates in the 
discrete $H^2$ seminorm. 
Moreover, we observe higher convergence orders for the error in the 
$L^2$ norm.
The use of the refinement indicator $\eta$ in an adaptive mesh refining
algorithm is observed to lead to optimal convergence rates
for singular solutions. The theoretical justification of the error 
estimator exceeds the scope of this article and is left for future
research.

\bibliography{pkuthss}
\bibliographystyle{plain}
\end{document}